\documentclass{amsart}
\usepackage{amssymb,latexsym}
\usepackage{amsmath,amsfonts,amsthm}
\usepackage[dvips]{graphicx}
\usepackage[all]{xy}
\theoremstyle{plain}
\newtheorem{theorem}{Theorem}[section]
\newtheorem{prop}[theorem]{Proposition}
\newtheorem{lemma}[theorem]{Lemma}

\newtheorem{definition}[theorem]{Definition}
\newtheorem{cor}[theorem]{Corollary}
\theoremstyle{remark}
\newtheorem{remark}[theorem]{Remark}
\newtheorem{example}[theorem]{Example}
\def\co{\colon\thinspace}

\def\ep{\epsilon}
\begin{document}
\begin{abstract}
Using the Oh--Schwarz spectral invariants and some arguments of Frauenfelder, Ginzburg, and Schlenk, we show that the $\pi_1$-sensitive Hofer--Zehnder capacity of any subset of a closed symplectic manifold is less than or equal to its displacement energy.  This estimate is sharp, and implies some new extensions of the Non-Squeezing Theorem. 
\end{abstract}
\title{The sharp energy-capacity inequality}
\author{Michael Usher}
\address{Department of Mathematics, University of Georgia, Athens, GA 30602}
\email{usher@math.uga.edu}
\maketitle
\section{Introduction}
Let $(M,\omega)$ be a symplectic manifold.  Any compactly supported smooth function $H\co (\mathbb{R}/\mathbb{Z})\times M\to\mathbb{R}$ (with support contained in $(\mathbb{R}/\mathbb{Z})\times (M\setminus \partial M)$ if $\partial M\neq \varnothing$) then induces a flow $\phi_{H}^{t}\co M\to M$ ($0\leq t\leq 1$), obtained by integrating the time-dependent vector field $X_H$ defined by $d(H(t,\cdot))=\iota_{X_H}\omega$.   The diffeomorphisms $\phi_{H}^{t}$ obtained in this manner are known as Hamiltonian symplectomorphisms; these form a subgroup, denoted $Ham^c(M,\omega)$, of the group of symplectomorphisms of $(M,\omega)$.  

If $A\subset M$ is any subset, there are various ways of measuring some version of the ``symplectic size'' of $A$ by considering how $A$ interacts with elements of $Ham^c(M,\omega)$.  Two of these, the \emph{Hofer--Zehnder capacity} $c_{HZ}(A)$ and its sibling the $\pi_1$-sensitive Hofer--Zehnder capacity $c_{HZ}^{\circ}(A,M)$, are defined as follows.  Let \[ \mathcal{H}(A)=\left\{H\in C^{\infty}(M)\left|\begin{array}{ll} supp(H)\Subset A\setminus \partial M, \,\,0\leq H\leq \max H,\, H^{-1}(\{0\})\mbox{ and }\\H^{-1}(\{\max H\})\mbox{ both contain nonempty open sets}\end{array}\right.\right\}.\] (The symbol $\Subset$ means ``is contained in a compact subset of.'') Call $H\in \mathcal{H}(A)$ \emph{HZ-admissible} if the Hamiltonian flow $\phi_{H}^{t}$ of the $t$-independent extension of $H$ to a function on $(\mathbb{R}/\mathbb{Z})\times M$ has no nonconstant periodic orbits of period at most $1$.  
Similarly, call $H\in \mathcal{H}(A)$ \emph{HZ$^{\circ}$-admissible} if $\phi_{H}^{t}$ has no nonconstant
periodic orbits of period at most $1$ which are contractible in $M$.  Now define \[ c_{HZ}(A)=\sup\{\max H|H\in \mathcal{H}(A)\mbox{ is HZ-admissible}\}\] and \[ c_{HZ}^{\circ}(A,M)=\sup\{\max H|H\in \mathcal{H}(A)\mbox{ is HZ$^{\circ}$-admissible}\}.\]  The notation is consistent with the fact that $c_{HZ}(A)$ depends only on $(A,\omega|_A)$ (and on $\dim M$), whereas 
whether an orbit of $\phi_{H}^{t}$ is contractible in $M$ or not of course depends on $M$, and therefore $c_{HZ}^{\circ}(A,M)$ depends on the ambient manifold $M$ and not just on $A$.  Obviously any HZ-admissible Hamiltonian is HZ$^{\circ}$-admissible, so \[ c_{HZ}(A)\leq c_{HZ}^{\circ}(A,M).\]

Hofer and Zehnder showed in \cite{HZ} that $c_{HZ}$ is a symplectic capacity in what is now the usual sense of the term (see, \emph{e.g.}, Chapter 2 of \cite{HZbook}), and that, where $B^{2n}(r)$ is the ball of radius $r$ in $\mathbb{R}^{2n}$, one has $c_{HZ}(B^{2n}(r))=\pi r^2$.  As a more or less immediate consequence of this latter fact together with the naturality of the definition of $c_{HZ}$, one has \[ w_G(A)\leq c_{HZ}(A)\] for any $A$, where the \emph{Gromov width} $w_G(A)$ is defined as the supremum of the numbers $\pi r^2$ over all $r$ with the property that the ball $B^{2n}(r)$ embeds symplectically in $A$.

The other way of measuring the size of $A$ using $Ham^c(M,\omega)$ that we shall discuss is by means of the \emph{displacement energy} $e(A,M)$.  To define $e(A,M)$, we recall first the definition of the Hofer norm of a Hamiltonian symplectomorphism $\phi\in Ham^c(M,\Omega)$ \cite{H}.  Introduce the notation $H\mapsto \phi$ to signify that a compactly supported smooth function $H\co (\mathbb{R}/\mathbb{Z})\times M\to \mathbb{R}$ has the property that $\phi_{H}^{1}=\phi$.  The Hofer norm of $\phi$ is then \[ \|\phi\|=\inf_{H\mapsto \phi}\left\{\int_{0}^{1}\left(\max_{p\in M}H(t,p)-\min_{p\in M}H(t,p)\right)dt\right\}.\]
It is a subtle fact, proven in \cite{H} for $M=\mathbb{R}^{2n}$ and in \cite{LM} for an arbitrary symplectic manifold $(M,\omega)$, that $\|\phi\|$ is positive unless $\phi$ is the identity map.  Now we can define the displacement energy of $A$ (in $M$) as \[ e(A,M)=\inf\{\|\phi\| |\phi(A)\cap A=\varnothing\}\] if $A$ is compact and as \[ e(A,M)=\sup\{e(K,M)|K\Subset A\}\] for general $A$.  

We can now state the main result of this paper.
\begin{theorem}\label{main} If $(M,\omega)$ is closed then for any $A\subset M$ we have \[ c_{HZ}^{\circ}(A,M)\leq e(A,M).\]
\end{theorem}

It is easy to see that the above inequality is sharp, since, writing $2n=\dim M$, if $M$ contains a Darboux ball $U$ of radius $3r$ then for any $\ep>0$ it is straightforward to explicitly construct $\phi\in Ham^c(U,\omega|_U)$ which displaces a copy of $B^{2n}(r)$ and has energy at most $\pi r^2+\ep$; thus $e(B^{2n}(r),M)=c_{HZ}^{\circ}(B^{2n}(r),M)=\pi r^2$.

Theorem \ref{main} was proven in Chapter 5 of \cite{HZbook}  with the closed manifold $(M,\omega)$ replaced by $\mathbb{R}^{2n}$ with its standard symplectic structure (and earlier in \cite{H} with $M=\mathbb{R}^{2n}$ and $c_{HZ}$  replaced by the Ekeland-Hofer capacity), and is also known for $(M,\omega)$ equal to an (open) symplectically aspherical convex manifold (see Remark 1.7(3)(i) of \cite{FGS}).     A weaker inequality $c_{HZ}^{\circ}(A,M)\leq 4e(A,M)$ was proven in \cite{Sch} for all tame $(M,\omega)$, and in \cite{LM} it is shown that $w_G(A)\leq 2e(A,M)$ for \emph{any} symplectic manifold $(M,\omega)$.  

Given the results mentioned in the previous paragraph in which $M$ is taken to be noncompact, it is natural to wonder if Theorem \ref{main} can be extended to noncompact symplectic manifolds.  In view of the constructions of \cite{FS}, it seems plausible that our proof could be extended to any $M$ which is symplectomorphic outside a compact set to the complement of a compact set in a product of manifolds each of which is either closed or convex, the point being that in such a context there is a maximum principle which ensures that Hamiltonian Floer theory is well-behaved for appropriately chosen Hamiltonians.  However, an extension to all tame symplectic manifolds would almost certainly require fundamentally different methods.  

Note that Theorem \ref{main} quickly implies at least a few analogues of itself in which $M$ is replaced by a noncompact manifold.  Let us say that the $2n$-dimensional symplectic manifold $(M,\omega)$ is of Type (C) if $M$ can be written as $\cup_{i=1}^{\infty}K_i$ where $\{K_i\}_{i=1}^{\infty}$ is an increasing sequence of codimension-zero compact sets with the property that, for each $i$, there exists a closed $2n$-dimensional symplectic manifold $(X_i,\omega_i)$ into which $K_i$ symplectically embeds.  If additionally the embeddings $\iota_i\co K_i\hookrightarrow X_i$ may be arranged to have the property that $\ker(\iota_{i*}\co \pi_1(K_i)\to \pi_1(X_i))$ is contained in $\ker(\pi_1(K_i)\to \pi_1(M))$, let us say that $(M,\omega)$ is of Type (C$^{\circ}$).  A moment's thought about the definitions of $c_{HZ}^{\circ}$ and $e$ shows then that (given Theorem \ref{main} for closed manifolds) Theorem \ref{main} holds equally well for any $(M,\omega)$ which is of Type (C$^{\circ}$), and also that if $(M,\omega)$ is only of Type (C) then one still has an inequality $c_{HZ}(A)\leq e(A,M)$ for subsets $A\subset M$.  Obviously,  $\mathbb{R}^{2n}$ and $T^*S^1$ are of Type (C$^{\circ}$)
(in the first case take $K_i=[-i,i]^{2n}$ and $X_i=\mathbb{R}^{2n}/(3i\mathbb{Z})^{2n}$, and in the second take $K_i=S^1\times [-i,i]$ and $X_i=\mathbb{R}^2/(\mathbb{Z}\oplus 3i\mathbb{Z})$).  More interestingly, Theorem 3.2 of \cite{LiMa} shows that every Stein manifold is of Type (C), and any convex symplectic $4$-manifold is of Type (C) by a famous theorem of Eliashberg and Etnyre \cite{El}, \cite{Et}.  This author does not know whether every symplectic manifold of Type (C) is also of Type (C$^{\circ}$). Of course, the class of manifolds of Type (C) (and likewise that of manifolds of Type (C$^{\circ}$)) is closed under Cartesian products, so we have for instance:
\begin{cor} \label{typec} If $(M,\omega)$ is the Cartesian product of any finite collection of closed symplectic manifolds, Stein manifolds, and/or convex symplectic $4$-manifolds (or more generally if $(M,\omega)$ is any symplectic manifold of Type (C)), then the inequality \[ c_{HZ}(A)\leq e(A,M)\] holds for all 
subsets $A\subset M$.
\end{cor}

We also have the following improvement of the stable energy-capacity inequality of \cite{Sch}:
\begin{cor}\label{stable} If $(M,\omega)$ is of Type (C$^{\circ}$) then, for every $A\subset M$, we have \[ c_{HZ}^{\circ}(A,M)\leq e(A\times S^1,M\times T^*S^1).\]
\end{cor}

\begin{proof}  As noted earlier, $T^*S^1$ is of Type (C$^{\circ}$), so $M\times T^*S^1$ is as well and so Theorem \ref{main} implies that $c_{HZ}^{\circ}(A\times U_{\ep},M\times T^*S^1)\leq e(A\times U_{\ep},M\times T^*S^1)$ for every $\ep>0$, where $U_{\ep}=\{(\theta,v)\in T^*S^1||v|\leq \ep\}$.  The argument given in Section 2.2 of \cite{Sch} then immediately proves the corollary.
\end{proof}

\subsection{Consequences of Theorem \ref{main} for symplectic embedding problems}

Several interesting applications of the stable energy-capacity inequality to Hamiltonian dynamics are given in \cite{Sch}; dynamical applications along those lines generally only rely on using the inequality to show that some set $A$ has $c_{HZ}(A)<\infty$, and obviously our improvement of the inequality to Corollary \ref{stable} is no better for this purpose than the inequality $c_{HZ}^{\circ}(A,M)\leq 4e(A\times S^1,M\times T^*S^1)$ which is proven in \cite{Sch}.

On the other hand, the sharp nature of Theorem \ref{main} does make it useful for establishing various elaborations of the Non-Squeezing Theorem. In this direction, let us make the following two simple observations.

\begin{prop} \label{obvious} Let $(M,\omega)$ be a symplectic manifold and let $A\subset M$.  \begin{itemize}
\item[(i)] If $(P,\Omega)$ is another symplectic manifold then \[ e(A\times P,M\times P)\leq e(A,M).\]
\item[(ii)] If $(P,\Omega)$ is a \emph{closed} symplectic manifold then \[ c_{HZ}^{\circ}(A\times P,M\times P)\geq c_{HZ}^{\circ}(A,M)\] and \[ c_{HZ}(A\times P)\geq c_{HZ}(A).\]\end{itemize}
\end{prop}

\begin{proof}
Let $K\subset A\times P$ be any compact subset.  In particular we have $K\subset K_1\times K_2$ where $K_1\Subset A$ and $K_2\Subset P$.  If $e>e(A,M)$, there is then a Hamiltonan $H\co (\mathbb{R}/\mathbb{Z})\times M\to \mathbb{R}$ such that $\int_{0}^{1}\left(\max_M H(t,\cdot)-\min_M H(t,\cdot)\right)dt<e$ and $\phi_{H}^{1}(K_1)\cap K_1=\varnothing$.  If $\chi\co P\to [0,1]$ is any compactly supported smooth function with $\chi|_{K_2}=1$, then defining $\tilde{H}\co (\mathbb{R}/\mathbb{Z})\times M\times P\to\mathbb{R}$ by $\tilde{H}(t,m,p)=H(t,m)\chi(p)$ gives a Hamiltonian flow $\phi_{\tilde{H}}^{t}$ with $\phi_{\tilde{H}}^{1}(K_1\times K_2)\cap(K_1\times K_2)=\varnothing$.  Thus $e(K,M\times P)\leq e$.  So since $e>e(A,M)$ and $K\Subset A\times P$ were arbitrary this shows that $e(A\times P,M\times P)\leq e(A,M)$.

For the statement about $c_{HZ}^{\circ}$ in (ii), simply note that if $supp H\Subset A$ and $\phi_{H}^{t}$ has no nonconstant contractible periodic orbits of period at most $1$, then defining $\tilde{H}\co M\times P\to\mathbb{R}$ by $\tilde{H}(m,p)=H(m)$ results in a Hamiltonian supported in $A\times P$ with no nonconstant contractible periodic orbits of period at most $1$.  The same reasoning with the word ``contractible'' deleted proves the statement about $c_{HZ}$.
\end{proof}

\begin{cor}\label{nosqueeze} Suppose that $(M,\omega_M)$, $(N,\omega_N)$, $(P,\omega_P)$, and $(Q,\omega_Q)$ are symplectic manifolds, with $M$ closed and $N,P,Q$ all of Type (C), and $\dim M+\dim P=\dim N+\dim Q$.  Suppose that $A\subset P$ and $B\subset Q$ both have nonempty interior, and that there exists a symplectic embedding \[ M\times A\hookrightarrow N\times B.\]  Then \[ c_{HZ}(A)\leq e(B,Q).\]
\end{cor}
\begin{proof}Indeed, by a standard property of $c_{HZ}$ (or any other symplectic capacity), if the embedding were to exist then we would require $c_{HZ}(M\times A)\leq c_{HZ}(N\times B)$.  But Corollary \ref{typec} and Proposition \ref{obvious}(i)  show that $c_{HZ}(N\times B)\leq e(N\times B,N\times Q)\leq e(B,Q)$, while Proposition \ref{obvious}(ii) shows that (since we assume $M$ to be closed) $c_{HZ}(A)\leq c_{HZ}(M\times A).$
\end{proof}
\begin{example} If $M=\{pt\}$, $A=B^{2n}(r)$, $N=\mathbb{R}^{2n-2}$, and $B=B^{2}(R)$, this recovers the original Non-Squeezing Theorem (from \cite{G}), namely that $B^{2n}(r)$ symplectically embeds in $\mathbb{R}^{2n-2}\times B^2(R)$ only when $r\leq R$.
More generally, we find that for any symplectic manifold $N$ of Type (C) (again, this includes all products of closed, Stein, and/or four-dimensional convex symplectic manifolds), $B^{2n}(r)$ embeds in $N\times B^{2}(R)$ only when $r\leq R$.  This is the generalized Non-Squeezing Theorem 1.4 of \cite{LM}, except that in \cite{LM} it is proven for arbitrary symplectic $(2n-2)$-manifolds $N$, rather than just ones of Type (C).
\end{example}

\begin{example} If one allows the manifold $M$ in Corollary \ref{nosqueeze} to be nontrivial, one can get some results that, as far as the author knows, are new (the closest analogue in the literature seems to be a theorem in \cite{LP}, a similar version of which appears as Corollary 1.30 of \cite{Lu}, which establishes Corollary \ref{nosqueeze} when $A$ and $B$ are balls and $N$ splits as a product $M\times Y$ with $M$ as one factor).  For example, if $\ep>0$ and $g$ is a natural number, let $M=\Sigma_{g,\ep}$ be the closed symplectic $2$-manifold of genus $g$ and area $\ep$.  Then taking, say, $A=B^{2}(r)^{2}$ and $B=B^2(R)$, one  finds that, for any symplectic $4$-manifold $N$ of Type (C) and any (even very small) $\ep>0$, an embedding of $\Sigma_{g,\ep}\times B^{2}(r)^2\hookrightarrow N\times B^2(R)$ can exist only if $r\leq R$.  Contrastingly, the $h$-principle implies that for any $r>1$ there is $\ep>0$ such that $B^2(\ep)\times B^{2}(r)^2$ embeds symplectically in $B^2(1)^3$ (this follows from, \emph{e.g.}, the argument in \cite{Sbook}, 1.3.1).

Note that nonsqueezing results along these lines cannot be established using the Gromov width rather than the Hofer--Zehnder capacity, since $\Sigma_{g,\ep}\times B^2(r)^2$ has very small Gromov width when $\ep$ is very small.

\end{example}

\subsection{Summary of the proof of Theorem \ref{main}}

The proof of Theorem \ref{main} is heavily influenced by the paper \cite{FGS}, in which it is shown that the inequality $c_{HZ}^{\circ}\leq e$ follows as soon as there exists a function $H\mapsto \sigma(H)$ on the set of all (time-dependent) compactly supported Hamiltonians $H$ satisfying the properties of an ``action selector'' (see Section 1.1 of \cite{FGS} for a definition, though we will not actually use the notion).  We use the Oh-Schwarz spectral invariant $\rho(\cdot;1)$ (\cite{Oh1}, \cite{Schwarz}) in the role of this $\sigma$, but we should emphasize that we do not in fact show that this is an action selector: indeed, perhaps the most basic property of an action selector, namely the condition that $\sigma(H)$ always belong to the action spectrum of $H$, remains unestablished for $\rho(H;1)$ when $H$ is degenerate and $\omega$ is spherically irrational.  However, the properties of $\rho(H;1)$ are now fairly  well-understood when $H$ is nondegenerate, and (even though the Hamiltonians involved in the definition of $c_{HZ}$ are  quite degenerate) we are able to apply this understanding to push through enough of the arguments of \cite{FGS} with the hypotheses that Frauenfelder, Ginzburg, and Schlenk impose on $\sigma$ replaced by known properties of $\rho(\cdot;1)$.  Incidentally, as we explain in, respectively, Remarks \ref{specnorm} and \ref{minlength}, as byproducts of the proof of Theorem \ref{main} we obtain new proofs of the nondegeneracy of Oh's spectral norm (originally proven in \cite{Oh3}) and of Conjecture 1.2 of \cite{MSlim} (originally proven in \cite{Sch}).

The rest of the paper is organized as follows.  Section 2 recalls the definition of the Oh--Schwarz spectral invariants; a reader who is familiar with these might just skip to Theorem \ref{ohbackground}, where we collect most of the properties of the invariants that we will use.  The proof of Theorem \ref{main} is divided between Sections 3 and 4; in the former we bound $e(A,M)$ from below using an argument similar to one used in \cite{FGS}, while in the latter we bound $c_{HZ}^{\circ}(A,M)$ from above by refining a reparametrization trick that was employed in \cite{Sch} in order to improve a theorem from \cite{Oh2} (recalled as Theorem \ref{ohlength} below).

\subsection*{Acknowledgement} I am grateful to Y.-G. Oh for his feedback on a draft of this paper.

\section{Spectral invariants} Assume throughout the rest of the paper that $(M,\omega)$ is any closed symplectic manifold.  Define \[ \Gamma_{\omega}=\frac{\pi_2(M)}{\ker(\langle c_1,\cdot\rangle)\cap \ker(\langle [\omega],\cdot\rangle)}\] and introduce the Novikov ring \[   \Lambda_{\omega}=\big\{\sum_{g\in \Gamma_{\omega}}b_gg|b_g\in \mathbb{Q},(\forall C\in\mathbb{R})(\#\{g|b_g\neq 0,\int_{S^2}g^*\omega<C\}<\infty)\big\}.\]  The quantum cohomology of $(M,\omega)$, written $QH^*(M,\omega)$, is then equal as a group to $H^*(M,\mathbb{Q})\otimes \Lambda_{\omega}$.  Schwarz \cite{Schwarz} (in the symplectically aspherical case) and Oh (for general closed $(M,\omega)$) have constructed a map \[ \rho\co C^{\infty}((\mathbb{R}/\mathbb{Z})\times M)\times (QH^*(M,\omega)\setminus \{0\})\to\mathbb{R}.\]  The construction and properties of $\rho$ are nicely surveyed in \cite{Ohsurvey}; here we just summarize a few basic points.  

If $H\in C^{\infty}((\mathbb{R}/\mathbb{Z})\times M)$, write \[ \mathcal{P}^{\circ}_{H}=\{\gamma\co \mathbb{R}/\mathbb{Z}\to M|\dot{\gamma}(t)=X_H(\gamma(t)),\gamma\mbox{ is contractible}\},\] and \[ \tilde{\mathcal{P}}_{H}^{\circ}=\frac{\{(\gamma,w)|\gamma\in\mathcal{P}_{H}^{\circ}, w\co D^2\to M, \gamma(t)=w(e^{2\pi it})\}}{(\gamma,w)=(\gamma',w')\mbox{ if }\gamma=\gamma'\mbox{ and }\overline{w}\#w'\mbox{ is trivial in } \Gamma_{\omega}},\] where $\overline{w}\#w'$ is the sphere obtained by gluing $w$ and $w'$ orientation-reversingly along their common boundary $\gamma$.  The action functional $\mathcal{A}_H\co \tilde{\mathcal{P}}_{H}^{\circ}\to\mathbb{R}$ is then defined by \[ \mathcal{A}_H([\gamma,w])=-\int_{D^2}w^*\omega-\int_{0}^{1}H(t,\gamma(t))dt.\]  If one defines $\mathcal{A}_H$ using the same formula on an appropriate cover of the space of contractible loops in $M$, then one recovers $\tilde{\mathcal{P}}_{H}^{\circ}$ as the set of critical points of this extended functional.

Assume for the moment that $H\in C^{\infty}((\mathbb{R}/\mathbb{Z})\times M)$ is a nondegenerate Hamiltonian, in the sense that its time-$1$ map $\phi_{H}^{1}$ has graph which is transverse to the diagonal in $M\times M$. 
 One defines the Floer chain complex of $H$ as \[ CF_{*}(H)=\left\{\sum_{[\gamma,w]\in \tilde{\mathcal{P}}_{H}^{\circ}}a_{[\gamma,w]}[\gamma,w]\left|a_{[\gamma,w]}\in \mathbb{Q},(\forall C\in\mathbb{R})(\#\{[\gamma,w]|a_{[\gamma,w]}\neq 0,\mathcal{A}_H([\gamma,w])>C\}<\infty) \right.\right\};\] this is naturally a module over $\Lambda_{\omega}$.  Formally, one obtains the differential on $CF_*(H)$ by counting negative gradient flow lines for $\mathcal{A}_H$; see \cite{Sal} for a survey of the details of the construction for a large family of symplectic manifolds, and \cite{FO},\cite{LT} for the general case.  There is a natural isomorphism \[ \Phi\co QH^*(M,\omega)\to HF_*(H)\] from the quantum cohomology of $(M,\omega)$ to  the homology $HF_*(H)$ of the complex $CF_*(H)$.  
 
 If $c=\sum_{[\gamma,w]}a_{[\gamma,w]}[\gamma,w]\in CF_*(H)$, define \[ \ell(c)=\max\{\mathcal{A}_H([\gamma,w])|a_{[\gamma,w]}\neq 0\}.\]  Now if $a\in QH^*(M,\omega)\setminus \{0\}$ (and $H$ is nondegenerate) define \[ \rho(H;a)=\inf\{\ell(c)|[c]=\Phi(a)\},\] where if $c\in CF_*(H)$ we use $[c]$ to denote the homology class of $c$ in $HF_*(H)$.  It follows quickly from the properties of the natural isomorphism $\Phi$ that $\{\ell(c)|[c]=\Phi(a)\}$ is bounded below, so $\rho(H;a)$ is a real number and not $-\infty$.  
 
This defines $\rho(H;a)$ for $H$ nondegenerate; for more general $H$ one observes that, if $H_1$ and $H_2$ are nondegenerate Hamiltonians and $a\in QH^*(M;\omega)$ there is an estimate \[ |\rho(H_1;a)-\rho(H_2;a)|\leq \int_{0}^{1}\max_{p\in M}|H_1(t,p)-H_2(t,p)|dt.\]  In particular $\rho(\cdot;a)$ is continuous with respect to the $C^0$-norm on nondegenerate Hamiltonians, so since nondegenerate Hamiltonians are dense in the space of all Hamiltonians we may define $\rho(\cdot;a)$ on all of $C^{\infty}((\mathbb{R}/\mathbb{Z})\times M)$ by extending continuously.  

This completes the definition of $\rho$; we now state some of its properties that we shall use.  For background, we should recall that if $H,K\in C^{\infty}((\mathbb{R}/\mathbb{Z})\times M)$ are Hamiltonians with Hamiltonian flows $\phi_{H}^{t}, \phi_{K}^{t}$, then the Hamiltonian $H\#K$ defined by \[ (H\#K)(t,p)=H(t,p)+K(t,(\phi_{H}^{t})^{-1}(p))\] has the property that its Hamiltonian flow is given by $\phi_{H\#K}^{t}=\phi_{H}^{t}\circ \phi_{K}^{t}$.  In particular, if we set $\bar{K}(t,p)=-K(t,\phi_{K}^{t}(p))$ then $\bar{K}\#K=K\#\bar{K}=0$.  Also, the \emph{action spectrum} of $H$ is by definition the set \[ \Sigma(H)=\{\mathcal{A}_H([\gamma,w])|[\gamma,w]\in \tilde{\mathcal{P}}_{H}^{\circ}\}.\] By Lemma 2.2 of \cite{Oh0}, $\Sigma(H)$ has Lebesgue measure zero.  Finally, a Hamiltonian $H$ is called \emph{normalized} if it holds that, for all $t\in\mathbb{R}/\mathbb{Z}$, $\int_{M}H(t,\cdot)\omega^n=0$ where $\dim M=2n$.

\begin{theorem}(\cite{Oh1},\cite{Ohsurvey},\cite{U})\label{ohbackground}  The function $\rho\co C^{\infty}((\mathbb{R}/\mathbb{Z})\times M)\times (QH^{*}(M,\omega)\setminus \{0\})\to\mathbb{R}$ has the following properties:\begin{itemize}\item[(i)] If $r\co \mathbb{R}/\mathbb{Z}\to\mathbb{R}$ is a smooth function then \[ \rho(H+r;a)=\rho(H;a)-\int_{0}^{1}r(t)dt.\]
\item[(ii)] $-\int_{0}^{1}\max_{p\in M}(H(t,p)-K(t,p))dt\leq \rho(H;a)-\rho(K;a) \leq -\int_{0}^{1}\min_{p\in M}(H(t,p)-K(t,p))dt$.
\item[(iii)] $-\int_{0}^{1}\max_{p\in M} H(t,p)dt\leq \rho(H;1)\leq -\int_{0}^{1}\min_{p\in M} H(t,p)dt$. 
\item[(iv)]  $\rho(H\#K;1)\leq \rho(H;1)+\rho(K;1)$.
\item[(v)] If $H$ is \emph{nondegenerate} then $\rho(H;a)\in \Sigma(H)$.
\item[(vi)] If $H$ and $K$ are two \emph{normalized} Hamiltonians which have the property that $\phi_{H}^{1}=\phi_{K}^{1}$ and if the paths $\phi_{H}^{t}$ and $\phi_{K}^{t}$ in $Ham^c(M,\omega)$ are homotopic with fixed endpoints, then \[ \rho(H;a)=\rho(K;a).\]\end{itemize}
\end{theorem}    

\begin{proof} (i) is an obvious consequence of the definition, since adding a function of time to $H$ does not affect the Hamiltonian flow or the Floer trajectories but does affect the action functional $\mathcal{A}_H$ in a manner consistent with (i).

(ii) follows directly from Theorem 5.3 (2) of \cite{Oh1}.  (Strictly speaking in \cite{Oh1} the only Hamiltonians which are considered are the normalized ones, but it is easy to see that the proof of Theorem 5.3 (2) applies equally well to non-normalized Hamiltonians).

(iii) is (ii) with $K=0$, bearing in mind that (by, \emph{e.g.}, Theorem 6.1(2) of \cite{Oh1}) $\rho(0;1)=0$.

(iv) is a special case of the triangle inequality  $\rho(H\#K;a\cdot b)=\rho(H;a)+\rho(K; b)$ (Theorem 6.1 (4) of \cite{Oh1}) where $a\cdot b$ is the quantum product of $a,b\in QH^*(M,\omega)$ (one can use (i) to extend the triangle inequality from normalized Hamiltonians to non-normalized Hamiltonians).

(v) is Corollary 1.5 of \cite{U}; there is also a proof in \cite{OhCerf} in the strongly semipositive case.

In view of (v), (vi) follows from Theorem 6.1 of \cite{Ohsurvey}.
\end{proof}

Note in particular that, since $\min_{p\in M}\bar{H}(t,p)=-\max_{p\in M}H(t,p)$, (iii) above shows that \[ \rho(H;1)+\rho(\bar{H};1)\leq \int_{0}^{1}\left(\max_{p\in M}H(t,p)-\min_{p\in M}H(t,p)\right)dt.\]  So in the notation of the introduction we have \begin{equation}\label{diprho} \|\phi\|\geq \inf_{H\mapsto \phi}\left(\rho(H;1)+\rho(\bar{H};1)\right)\end{equation}
for any $\phi\in Ham^c(M,\omega)$.  

\begin{remark}\label{specnorm} In \cite{Oh3}, Oh defines a ``spectral norm'' $\gamma\co Ham^c(M,\omega)\to [0,\infty)$ by \[ \gamma(\phi)=\inf_{H\mapsto \phi}\left(\rho(H;1)+\rho(\bar{H};1)\right).\]  Our proof of Theorem \ref{main} will in fact show that if $U$ is an open set with the property that $\phi(U)\cap U=\varnothing$, then $\gamma(\phi)\geq c_{HZ}^{\circ}(U,M)$ (indeed, this follows immediately from Propositions \ref{kkbar} and \ref{aut} below).  Of course, if $\phi$ is not equal to the identity, there is a Darboux ball $U=B^{2n}(r)$ in $M$ such that  $\phi(U)\cap U=\varnothing$, and so we have $\gamma(\phi)\geq \pi r^2>0$.  Thus our argument gives a new proof of the fact that $\gamma$ vanishes only for the identity symplectomorphism; this was established by a subtle, rather different argument as Theorem A(1) of \cite{Oh3}.  The nondegeneracy of $\gamma$ is also proven as part of Theorem 12.4.4 of \cite{MS}, using a similar argument to the one in \cite{Oh3}.

\end{remark}

\section{Bounding $e$ from below}
The proof of Theorem \ref{main}, which is patterned after arguments in Section 2 of \cite{FGS}, follows from two propositions.  The first of these is:

\begin{prop}\label{kkbar} (compare \cite{FGS}, Proposition 2.1) Suppose that $H,K\co (\mathbb{R}/\mathbb{Z})\times M\to\mathbb{R}$ are two Hamiltonians, with the property that, where $S(H)=\cup_{t\in\mathbb{R}/\mathbb{Z}}supp(H(t,\cdot))\subset M$, we have \[ \phi_{K}^{1}(S(H))\cap S(H)=\varnothing.\]  Then \[ \rho(H;1)\leq \rho(K;1)+\rho(\bar{K};1).\]
\end{prop}
\begin{proof}
First we observe that it suffices to prove the proposition in the special case that $K$ is nondegenerate.\footnote{On the other hand, one can't assume $H$ is nondegenerate, since if $H$ were nondegenerate we would have $S(H)=M$ and no $K$ could satisfy the hypothesis.  Since we therefore can't apply Theorem \ref{ohbackground} (v) directly to $H$, this leads to some small differences between our proof and that of the corresponding result in \cite{FGS}.}  Indeed, the condition that $\phi_{K}^{1}(S(H))\cap S(H)=\varnothing$ is an open condition on $\phi_K$, so we can find a sequence of nondegenerate Hamiltonians $K_n$ which converge to $K$ in $C^2$-norm and satisfy $\phi_{K_n}^{1}(S(H))\cap S(H)=\varnothing$.  If the proposition is proven with $K$ replaced by the nondegenerate Hamiltonians $K_n$, then it also holds for $K$ by virtue of the continuity of $\rho(\cdot;1)$.

Assume therefore that $K$ is nondegenerate.  By Theorem \ref{ohbackground} (i) and the definition of $\bar{K}$, adding a function of time to $K$ leaves $\rho(K;1)+\rho(\bar{K};1)$ unchanged, so there is no loss of generality in assuming that $K$ is also normalized.  Then since $\phi_{K}^{1}$ displaces $S(H)$, the fixed points of $\phi_{H}^{1}\circ \phi_{K}^{1}$ are precisely the same as those of $\phi_{K}^{1}$, and all lie outside of $S(H)$.  So the definition of nondegeneracy implies that $H\#K$ is also nondegenerate.  

For each $t$, let $c(t)=\frac{\int_M H(t,\cdot)\omega^n}{\int_M\omega^n}$, so that the Hamiltonian $\underline{H}(t,p):=H(t,p)-c(t)$ is normalized.

Choose a smooth function $\chi\co [0,1/2]\to [0,1]$ with the properties that $\chi'(t)\geq 0$ for all $t$, $\chi(0)=0$, $\chi(1/2)=1$,  and $\chi'$ vanishes to infinite order at  $0$ and at $1/2$.  Now, for each $u\in [0,1]$, define \[ L_u(t,p)=\left\{\begin{array}{ll}\chi'(t)K(\chi(t),p) & 0\leq t\leq 1/2 \\ u\chi'(t-1/2)\underline{H}(\chi(t-1/2),p) & 1/2\leq t\leq 1\end{array}\right.\]  Since $\chi'$ vanishes to infinite order at $0$ and $1/2$, each $L_u$ defines a smooth Hamiltonian $(\mathbb{R}/\mathbb{Z})\times M\to\mathbb{R}$, which moreover is normalized since both $K$ and $\underline{H}$ are.  The time-$1$ flow of $L_u$ is given by $\phi_{L_u}^{1}=\phi_{u\underline{H}}^{1}\circ \phi_{K}^{1}$; thus, again, the fixed points of $\phi_{L_u}^{1}$ are just those of $\phi_{K}^{1}$, and $\phi_{L_u}^{1}$ coincides with $\phi_{K}^{1}$ on a neighborhood of each of these fixed points.  Thus the Hamiltonians $L_u$ are nondegenerate.  

\begin{lemma} $\rho(L_1;1)=\rho(L_0;1)+\int_{0}^{1}c(t)dt$.\end{lemma}
\begin{proof}
If $p$ is a fixed point of $\phi_{L_{u}}^{1}$, we have $\phi_{u\underline{H}}^{1}(\phi_{K}^{1}(p))=p$.  $p$ of course then cannot lie in $S(H)$, since $\phi_{K}^{1}$ displaces $S(H)$ and $\phi_{u\underline{H}}^{1}$ preserves $M\setminus S(H)$.  So since $\phi_{u\underline{H}}^{1}(S(H))=S(H)$, $\phi_{K}^{1}(p)$ cannot belong to $S(H)$.  Thus, regardless of the choice of $u\in [0,1]$, the $1$-periodic orbits of $X_{L_{u}}$ are precisely orbits of form \[ \gamma_p(t)=\left\{\begin{array}{ll}\phi_{K}^{\chi(t)}(p) & 0\leq t\leq 1/2 \\ p & 1/2\leq t\leq 1\end{array}\right.,\] where $p$ is any fixed point of $\phi_{K}^{1}$.  If $\gamma_p$ is any such orbit which is contractible, and if $w\co D^2\to M$ satisfies $w(e^{2\pi it})=\gamma_p(t)$, we then have \begin{align*} \mathcal{A}_{L_u}([\gamma_p,w])&=-\int_{D^2}w^*\omega-\int_{0}^{1}L_u(t,\gamma_p(t))dt
\\&=-\int_{D^2}w^*\omega-\int_{0}^{1/2}\chi'(t)K(\chi(t),\phi^{\chi(t)}_{K}(p))dt-\int_{1/2}^{1}u\chi'(t-1/2)(-c(\chi(t-1/2)))dt\\&=
-\int_{D^2}w^*\omega-\int_{0}^{1}K(t,\phi_{K}^{t}(p))dt+u\int_{0}^{1}c(t)dt.\end{align*}  The sum of the first two terms above is just equal to $\mathcal{A}_K([\gamma'_p,w'])$ where $\gamma'_p(t)=\phi_{K}^{t}(p)$ and $w'$ is a reparametrization of $w$ having $w'(e^{2\pi it})=\gamma'_p(t)$.  We therefore obtain that \[ \Sigma(L_u)=\left\{\alpha+u\int_{0}^{1}c(t)dt\left|\alpha\in \Sigma(K)\right.\right\}.\]  Therefore, by parts (ii) and (v) of Theorem \ref{ohbackground} and by the nondegeneracy of $L_u$, the function $u\mapsto \rho(L_u;1)-u\int_{0}^{1}c(t)dt$ is a continuous function from $[0,1]$ to the measure-zero set $\Sigma(K)$ and so is constant, from which the lemma follows.
\end{proof}

Now $\phi_{L_0}^{t}=\left\{\begin{array}{ll}\phi_{K}^{\chi(t)}& 0\leq t\leq 1/2\\ \phi_{K}^{1}& 1/2\leq t\leq 1\end{array}\right.$ is obviously homotopic rel endpoints to $\phi_{K}^{t}$ in $Ham^{c}(M,\omega)$, so by Theorem \ref{ohbackground} (vi) we have \[ \rho(K;1)=\rho(L_0;1).\]  Likewise \[ \phi_{L_1}^{t}=\left\{\begin{array}{ll}\phi_{K}^{\chi(t)} & 0\leq t\leq 1/2 \\ \phi_{\underline{H}}^{\chi(t-1/2)}\circ \phi_{K}^{1} & 1/2\leq t\leq 1\end{array}\right. \] is homotopic rel endpoints to $\phi_{\underline{H}}^{t}\circ \phi_{K}^{t}=\phi_{\underline{H}\#K}^{t}$, so (using that $L_1$ and $\underline{H}\#K $ are, unlike $H\#K$, normalized), Theorem \ref{ohbackground} (vi) shows that \[ \rho(\underline{H}\#K;1)=\rho(L_1;1)=\rho(K;1)+\int_{0}^{1}c(t)dt.\]

Hence \begin{align*} \rho(H;1)&=\rho(\underline{H};1)-\int_{0}^{1}c(t)dt= \rho((\underline{H}\#K)\#\bar{K};1)-\int_{0}^{1}c(t)dt
\\&\leq \rho(\underline{H}\#K;1)+\rho(\bar{K};1)-\int_{0}^{1}c(t)dt\\&=\rho(K;1)+\rho(\bar{K};1),\end{align*} where the first equality  uses Theorem \ref{ohbackground} (i) and the inequality uses Theorem \ref{ohbackground} (iv).
\end{proof}

\begin{cor}\label{dispcor}
If $A\subset M$, then for every $H\in C^{\infty}((\mathbb{R}/\mathbb{Z})\times M)$ satisfying $S(H)\subset A$ we have \[ e(A,M)\geq \rho(H;1).\]
\end{cor}

\begin{proof}  If $S(H)\subset A$, then for any number $e>e(A,M)$ there is $\phi\in Ham^{c}(M,\omega)$ such that $\phi(A)\cap A=\varnothing$ and $\|\phi\|<e$.  So in light of (\ref{diprho}) there is $K$ such that $\phi_{K}^{1}=\phi$ and $\rho(K;1)+\rho(\bar{K};1)<e$.  So Proposition \ref{kkbar} shows that $\rho(H;1)< e$.

\end{proof}

\section{Bounding $c_{HZ}^{\circ}$ from above}

Theorem \ref{main} now follows quickly from the following fact, corresponding to Condition (AS2$^+$) in the framework set up in \cite{FGS} (but adjusted for a different sign convention for Hamiltonian vector fields):

\begin{prop}\label{aut} If $H\co M\to\mathbb{R}$ is an autonomous Hamiltonian whose Hamiltonian flow has no nonconstant contractible periodic orbits of period at most $1$, then \[ \rho(H;1)=-\min_M H.\] \end{prop}

\begin{proof}[Proof of Theorem \ref{main}, assuming Proposition \ref{aut}]  If $c< c_{HZ}^{\circ}(A,M)$, there is an autonomous Hamiltonian $H\in \mathcal{H}(A)$ whose Hamiltonian vector field $X_H$ has no nonconstant contractible periodic orbits of period at most $1$ and which satisfies $\max H\geq c$.  But then $X_{-H}$ also has no nonconstant contractible periodic orbits of period at most $1$ (since the periodic orbits of $X_H$ are obtained from those of $X_{-H}$ by time reversal), and we may apply Proposition \ref{aut} to conclude that $\rho(-H;1)=-\min_M (-H)=\max_M H\geq c$.  So Corollary \ref{dispcor} (applied to $-H$) shows that $c\leq e(A,M)$.\end{proof}

\begin{remark}\label{minlength}  Theorem III of \cite{Oh2}, together with Proposition \ref{aut} applied to both $H$ and $\bar{H}$, imply a conjecture of McDuff and Slimowitz \cite{MSlim}, namely that for any $H$ as in Proposition \ref{aut} the path $\phi_{H}^{t}$ minimizes the Hofer length within its homotopy class.  This conjecture was proven earlier in \cite{Sch}; it is not quite true that the proof arising from Proposition \ref{aut} is independent of Schlenk's proof, since our proof of Proposition \ref{aut} involves refining an argument that we learned from \cite{Sch}.  However, the principal ingredients in the respective proofs (Theorem 1.4 of \cite{MSlim} in the case of the proof in \cite{Sch}, and Theorem IV of \cite{Oh2} in our case) are established by rather different methods.
\end{remark}

The proof of Proposition \ref{aut} occupies the rest of this section.  The proposition generalizes (and its proof crucially uses) a result from \cite{Oh2}.  To state this result, we introduce some terminology, which is borrowed from \cite{Sch} (readers familiar with \cite{MSlim} should note that our definition of ``slow'' is different from the one there).

\begin{definition} \label{slowflat}  Let $(M,\omega)$ be a symplectic manifold, and let $H\co M\to\mathbb{R}$ be a smooth function, with Hamiltonian vector field $X_H$.  
\begin{itemize}\item[(i)] $H$ is called \emph{slow} if every contractible periodic orbit of $X_H$ having period at most $1$ is constant. \item[(ii)] $H$ is called \emph{flat} if, at every critical point $p$ of $H$, the linearized flow $(\phi_{t})_*\co T_pM\to T_p M$ has no nonconstant periodic orbits of period at most $1$.
\end{itemize}
\end{definition}

In this language, one has:
\begin{theorem}\label{ohlength}(\cite{Oh2}, Theorem IV)   If $H\co M\to\mathbb{R}$ is a slow, flat, Morse function, then \[ \rho(H;1)=-\min_M H.\]
\end{theorem}
(In fact, the result in \cite{Oh2} is somewhat more general than this, as it allows for the existence of orbits of period strictly less than $1$, both for the Hamiltonian flow $X_H$ and for the linearized flows at all of the critical points other than the global minimum and maximum, provided merely that these flows do not have periodic orbits of period exactly $1$.)

In order to generalize Theorem \ref{ohlength} we use the following approximation result:

\begin{theorem}\label{approx} Let $H\co M\to\mathbb{R}$ be any slow Hamiltonian with Hamiltonian vector field $X_H$ and  Hamiltonian flow $\{\phi_{H}^{t}\}_{0\leq t\leq 1}$.   Then if $\ep>0$ there is a smooth function $K\co M\to\mathbb{R}$ such that $\|K-H\|_{C^0}<\ep$,  and such that $K$ is a slow, flat, Morse function.
\end{theorem}

Except for the requirement that $K$ be Morse, this was essentially proven as part of the proof of Theorem 1.3 in \cite{Sch}; however the functions produced in \cite{Sch} have nonempty open sets consisting of critical points, and so are far from being Morse.  

\begin{proof}

Choose a background Riemannian metric on $M$, induced by an almost complex structure compatible with $\omega$.  
Let \[ F_n=\{p\in M||\nabla H(p)|\leq 1/n\}.\]  Note that if $s\in \cap_{n=1}^{\infty}H(F_n)$, then there are $p_n\in M$ such that $H(p_n)= s$ and $|\nabla H(p_n)|\leq 1/n$, so the compactness of $M$ and the continuity of $H$ and $\nabla H$ show that there is $p\in M$ with $H(p)=s$ and $\nabla H(p)=0$.  Thus $\cap_{n=1}^{\infty}H(F_n)$ is precisely the set of critical values of $H$, and so has Lebesgue measure zero by Sard's theorem.  So since the image of $H$ has finite Lebesgue measure, we have \[ m_{Leb}(H(F_n))\to 0 \mbox{ as }n\to\infty.\]

Let $\zeta>0$ be given.  There is then $N$  such that $m_{Leb}(H(F_N))<\zeta/2$.  Choose a smooth function $\psi\co M\to\mathbb{R}$ such that $0\leq \psi\leq 1$, $\psi|_{F_{2N}}=1$, and $supp(\psi)\subset F_N$

  Now (using that $F_N$ and hence $H(F_N)$ are compact) there are open intervals $I_1,\ldots,I_k$ with disjoint closures such that $H(F_N)\subset \cup_{j=1}^{k} I_k $ and $\sum_{j=1}^{k}m_{Leb}(I_j)<3\zeta/4$.  Choose open intervals $I'_1,\ldots I'_k$ with disjoint closures such that $\bar{I_j}\subset I'_j$ for each $j$ and $\sum_{j=1}^{k}m_{Leb}(I'_j)<\zeta$.  Write $S=\cup_{j=1}^{k}I_j$, $S'=\cup_{j=1}^{k}I'_j$.

  Write \[ B=\sup_{p\in M}|\nabla(\nabla H)(p)|\]  and let $\eta>0$ be a small number ($\eta$ will be further specified later; in particular it should be smaller than $1$).  Let $f\co [\min H,\max H]\to [\min H,\max H]$ be a $C^{\infty}$ function satisfying the following properties:
\begin{itemize}\item[(i)] $f(\min H)=\min H$
\item[(ii)] $0<f'(s)\leq 1$ for all $s$ 
\item[(iii)] $f'(s)=1$ if $s\notin S'$
\item[(iv)] $f'(s)=\frac{\eta}{2}$ if $s\in S$\end{itemize}

So since $m_{Leb}(S')<\zeta$, we have $0\leq s-f(s)< \zeta$ for all $s\in [\min H,\max H]$.  In particular, $\|f\circ H-H\|_{C^0}<\zeta$, independently of $\eta$.

If $q\in F_N$ (which in particular implies $f''(H(q))=0$ by the construction of $f$) one has \[ |\nabla(f\circ H)|=f'(H(q))|\nabla H(q)|\leq \frac{\eta}{2N} \] and \[ |\nabla(\nabla(f\circ H))(q)|=\frac{\eta}{2}|\nabla(\nabla H)(q)|\leq \frac{B\eta}{2}.\]  

Hence we have \[ \|X_{f\circ H}\|_{C^1(F_N)}\leq \frac{\eta}{2}+\frac{B\eta}{2}.\]  So recalling our above function $\psi\co M\to\mathbb{R}$ (which is supported in $F_{N}$), we obtain \[ \|\psi X_{f\circ H}\|_{C^1}\leq \left(\frac{\eta}{2}+\frac{B\eta}{2}\right)\|\psi\|_{C^1}.\]  By the Yorke estimate \cite{Y} (and the Whitney embedding theorem), there is $\eta_0>0$ such that  any vector field $V$ on $M$ with $\|V\|_{C^1}\leq \eta_0(B+1)\|\psi\|_{C^1}$ will have no nonconstant periodic orbits of period at most one.  So take $0<\eta<\eta_0$ and let $\{K_m\}_{m=1}^{\infty}$ be a sequence of Morse functions on $M$ such that $K_m\to f\circ H$ in $C^2$-norm.  Then $X_{K_m}\to X_{f\circ H}$ in $C^1$-norm, and so $\psi X_{K_m}\to \psi X_{f\circ H}$ in $C^1$-norm.  Hence, for $m$ large enough that \[ \|\psi X_{f\circ H}-\psi X_{K_m}\|_{C^1}<\left(\frac{(B+1)\eta}{2}\right)\|\psi\|_{C^1},\] the vector field $\psi X_{K_m}$ has no nonconstant periodic orbits of period at most one. 

Now since $H$ is slow and since  $0< f'\leq 1$, it quickly follows that $f\circ H$ is slow.  Thus all contractible periodic orbits of $X_{f\circ H}$ of period at most one   are constant orbits at critical points of $f\circ H$; in particular all such orbits are contained in the interior of $F_{2N}$.  An easy application of the Arzel\`a--Ascoli theorem then shows that, if $m$ is sufficiently large, any contractible periodic orbit of period at most one of $X_{K_m}$ must be contained in $F_{2N}$.  So (since $\psi|_{F_{2N}}=1$) for $m$ sufficiently large the only contractible periodic orbits of $X_{K_m}$ with period at most one are in fact periodic orbits of $\psi X_{K_m}$ and hence are constant by what we've already shown.

Thus $K_m$ is slow for $m$ sufficiently large.  Further, the fact that $K_m\to f\circ H$ in $C^2$ shows that all critical points $p$ of $K_m$ are, for $m$ large enough, contained in $F_N$, and at any such $p$ we have $|\nabla(\nabla (f\circ H))(p)|<B\eta/2$ and therefore $|\nabla(\nabla K_m)(p)|<B\eta$ for large enough $m$.  So (as long as $\eta$ has been taken smaller than $2\pi/B$) the linearized flow $\phi_{K_m}^{t}$ at any such critical point will have no nonconstant periodic orbits of period at most $1$.

Thus, for sufficiently large $m_0$, $K_{m_0}$ is a slow, flat, Morse function with $\|K_{m_0}-f\circ H\|_{C^2}<\zeta$.  Since $\|f\circ H-H\|_{C^0}<\zeta$, this proves the theorem (setting $\zeta=\ep/2$ and $K=K_{m_0}$). 

\end{proof}

Proposition \ref{aut} is now a quick consequence of Theorems \ref{ohlength} and \ref{approx}.  Indeed, if $H$ is as in Theorem \ref{aut} and if $\ep>0$, use Theorem \ref{approx} to find a slow, flat, Morse function $K$ with $\|K-H\|_{C^0}< \ep$.  By Theorem \ref{ohlength}, $\rho(K;1)=-\min_M K$.  Theorem \ref{ohbackground}(ii) shows that $|\rho(H;1)-\rho(K;1)|\leq \|H-K\|_{C^0}<\ep$, and of course one has $|\min_M H-\min_M K|\leq \|H-K\|_{C^0}< \ep$, and so we have \[ |\rho(H;1)+\min_M H|<2\ep.\]  $\ep>0$ was arbitrary, so this proves Proposition \ref{aut} and hence also Theorem \ref{main}.

\end{document}